\documentclass{amsart}
\usepackage{amsmath}
\usepackage{amssymb}
\usepackage{latexsym}
\usepackage{amsbsy}
\usepackage[all]{xy}
\usepackage{amscd}
\usepackage[dvips]{graphicx}

\newtheorem{theorem}{Theorem}[section]
\newtheorem{lemma}[theorem]{Lemma}

\theoremstyle{definition}

\newtheorem{Prop}[theorem]{Proposition}
\newtheorem{Cor}[theorem]{Corollary}

\newtheorem{remark}[theorem]{Remark}

\numberwithin{equation}{section}

\newcommand\bbn{{\mathbb N}}
\newcommand\bbz{{\mathbb Z}}
\newcommand\bbq{{\mathbb Q}}

\newcommand\aut{\mbox{Aut}}

\newcommand\ext{\mbox{Ext}\,}

\newcommand\Hom{\mathrm{Hom}}

\newcommand\udim{\mbox{\underline {dim}}}
\newcommand\ed{\mbox{End}}

\newcommand\mc{{\mathcal{C}}}

\newcommand\md{{\mathcal{D}}}

\newcommand\mood{{\text{mod}}}

\begin{document}

\title{Hall algebras for odd periodic triangulated categories}

\author{Fan Xu}
\author{Xueqing Chen}

\address{Department of Mathematics\\ Tsinghua University,
Beijing{\rm 100875},P.R.China} \email{fanxu@mail.tsinghua.edu.cn}

\address{Department of Mathematical and Computer Sciences\\University of Wisconsin--Whitewater\\
800 W. Main Street\\ Whitewater, WI. 53190. USA}
\email{chenx@uww.edu}

\thanks{Fan Xu was supported by
Alexander von Humboldt Stiftung and was also partially supported by
the Ph.D. Programs Foundation of Ministry of Education of China (No.
200800030058)}


\date{\today}


\begin{abstract}
We define the Hall algebra associated to any triangulated category
under some finiteness conditions with the $t$-periodic translation
functor $T$ for odd $t>1.$ This generalizes the results in
\cite{Toen2005} and \cite{XX2006}.
\end{abstract}

\maketitle

\section{Introduction}
Hall algebras provided a framework involving the categorification
and the geometrization of Lie algebras and quantum groups in the
past two decades (see \cite{Lusztig2000, Nakajima1998, PX2000,
Ringel1990, Xiao97}). In a broad sense, a Hall algebra provides a
tool allowing one to code a category.

Let {$\mathcal{A}$} be a finitary category, i.e., a (small) abelian
category satisfying: $(1)$ $|\Hom (M,N)| <\infty $; $(2)$ $|\ext^1
(M,N)| <\infty$ for any $M,N \in \mathcal{A}$. Some typical examples
of finitary categories are the category of finite length modules
over some discrete valuation ring $R$ whose residue field $R/m$ is
finite, the category of finite dimensional representations of a
quiver $\mathcal{Q}$ over a finite field $k$ and the category of
coherent sheaves on some projective scheme over a finite field $k$.

The Hall algebra $\mathcal{H}(\mathcal{A})$ associated to a finitary
category $\mathcal{A}$ is an associative algebra, which, as a
$\bbq$-vector space, has a basis consisting of the isomorphism
classes $[X]$ for $X\in \mathcal{A}$ and has the multiplication
$[X]*[Y]=\sum_{[L]}\mathrm{g}_{XY}^L[L]$, where $X,Y,L \in
\mathcal{A}$ and $\mathrm{g}_{XY}^L=|\{ M \subset L| M \simeq X
\text{ and } L/M \simeq Y\}|$ is the structure constant related to
counting exact sequences $0 \rightarrow X \xrightarrow{f} L
\xrightarrow{g}Y \rightarrow 0$ and is called the \emph{Hall number}
(see \cite{Joyce}). Equivalently,
$$\mathrm{g}_{XY}^L=\displaystyle\frac{|\mathcal{M}(X,L)_{Y}|}{|\aut X|},$$ where $\mathcal{M}(X,L)_{Y}$ is the subset of
$\textrm{Hom}(X,L)$ consisting of monomorphisms $f: X
\hookrightarrow L$ whose cokernels $\textrm{Coker}(f)$ are
isomorphic to $Y$. Indeed, $\mathrm{g}_{XY}^L$ can be calculated as
follows. Define $$ E(X,Y;L)=\{(f,g)\in \Hom(X,L)\times \Hom(L,Y)\mid
$$ $$ 0 \rightarrow X \xrightarrow{f} L \xrightarrow{g}Y \rightarrow
0 \mbox{ is an exact sequence}\}.$$ The group $\aut X\times \aut Y$
acts freely on $E(X,Y;L)$ and the orbit of $(f,g)\in E(X,Y;L)$ is
denoted by $(f,g)^{\wedge}:=\{(af,gc^{-1})\mid (a,c)\in \aut X\times
\aut Y\}.$ If the orbit space is denoted by $O(X,Y;L)=\{
(f,g)^{\wedge}|(f,g) \in E(X,Y;L)\}$, then
$\mathrm{g}_{XY}^L=|O(X,Y;L)|$. The associativity of the Hall
algebra of a finitary category follows from pull--back and push--out
constructions. One can refer to \cite{Schiffmann1} and
\cite{Schiffmann2} for systematic introductions to this topic.

In the following, we shall describe briefly the current developments
of the theory of Hall algebras. The term ``Hall Algebra'' is due to
Ringel, as the generalization of ``algebra of partitions''
originally constructed in the context of abelian $p$-groups by
Hall~\cite{Hall}, which even has a trace back to the problems
considered by Steineiz~\cite{Steinitz}. In the early 90's, Ringel
defined the Hall algebra associated to an abelian category
in~\cite{Ringel1990}. In particular, when the abelian category
$\mathcal{A}$ is the module category of finitely generated modules
of a finite dimensional hereditary algebra $\Lambda$ over a finite
field, the positive part of Drinfeld-Jimbo's quantum group $\bf{U}_q
(\mathfrak{g})$ of the Kac--Moody Lie algebra $\mathfrak{g}$ is
obtained as the (generic, twisted) Hall algebra of $\mathcal{A}$ if
$\Lambda$ and $\mathfrak{g}$ share the same diagram
(see~\cite{Ringel1990,Green}). The theory of Hall algebras is
closely related to representation theory and algebraic geometry.
Note that due to this link established between the geometry of
quiver representations and quantum groups that
Lusztig~\cite{Lusztig} developed his theory of canonical basis.
Kashiwara~\cite{Kashiwara} constructed independently such bases
(called the crystal basis) by combinatorial methods. In studying a
Hall algebra associated to the category of coherent sheaves on the
projective line, Kapranov~\cite{Kap1} opened a new direction in the
theory. One can extend his study to the case of elliptic curves and
of surfaces.

A Hall algebra was originally made from an abelian category. After
Ringel's discovery, as a generalization, some attempts to strengthen
the relationship between the Hall algebra over some abelian category
and the quantum group of some Lie algebra have led to the problem of
associating some kinds of Hall algebras to categories which are
triangulated rather than abelian. One may also naturally ask the
following question (see~\cite{Ringel1990}): how to recover the whole
Lie algebra and the whole quantum group? A direct way is to use the
reduced Drinfeld double to glue together two Borel parts as shown by
Xiao~\cite{Xiao97}. However, this construction is not intrinsic,
that is, not naturally induced by the module category of
corresponding hereditary algebra $\Lambda$. Therefore, one needs to
replace the module category by a larger category. It was first
pointed out by Xiao~\cite{Xiao95}, cf. also~\cite{Kap2}, that an
extension of the construction of the Hall algebra to the derived
category of a finite dimensional hereditary algebra $\Lambda$ (which
is also a triangulated category) might yield the whole quantum
group. Unfortunately, as Kapranov pointed out, a direct mimicking of
the Hall algebra construction, but with triangles replacing exact
sequences, fails to give an associative multiplication. Therefore,
we can not use this way to recover the whole quantum groups or
enveloping algebras. However, Peng and Xiao~\cite{PX2000} defined an
analogous multiplication of Hall multiplication over the
$2$-periodic $k$-additive triangulated category for some finite
field $k$ with the cardinality $q$. A triangulated category is
periodic if the translation functor $T=[1]$ is periodic, i.e.,
$[1]^d\cong 1$ for some $d\in \bbn$. Although this multiplication
is, in general, not associative, the induced Lie bracket given by
the commutator of the nature Hall multiplication satisfies Jacobi
identity. This well-defined Hall type algebra provides the
realization of Kac-Moody Lie algebras. In \cite{Xu2007}, we
constructed a new multiplication over $\bbz/(q-1)$ for the
$2$-periodic triangulated category which is almost associative for
isomorphism classes of indecomposable objects. This refines the
construct of Peng and Xiao~\cite{PX2000}.

On the other hand, Kapranov~\cite{Kap2} defined an associative
algebra called lattice algebra associated to the derived category of
any hereditary category. This associative algebra can be viewed as
an analog of the Hall algebra associated to a finitary category.
Essentially, to define a Hall algebra is to provide a formula for
Hall numbers. However, there is no formula for defining a Hall
algebra from any triangulated category. Recently,
To\"{e}n~\cite{Toen2005} made a remarkable development in this
direction. By using model category and the fiber products of model
categories, To\"{e}n defined a ``derived Hall algebra''--an
associative algebra associated to a triangulated category
$\mathcal{T}$ with the translation functor $[1]$ which appears as
homotopy category of model category whose objects are modules over a
sufficiently finitary dg (differential graded) category over a
finite field $k$. He obtained an explicit formula for the structure
constant $\Phi_{XY}^L$ of an associative multiplication on the
rational vector space generated by the isomorphism classes of
$\mathcal{T}$. The structure constant $\Phi_{XY}^L$ is called the
derived Hall number and supplied by a formula of To\"en
\cite[Proposition 5.1]{Toen2005} as follows:
$$\Phi_{XY}^L=\displaystyle\frac{|(X,L)_{Y}|}{|\mathrm{Aut}X|}\cdot
\big(\prod_{i>0}\displaystyle\frac{|\Hom(X[i],
L)|^{(-1)^i}}{|\Hom(X[i], X)|^{(-1)^i}}\big),$$ where $X, Y, L \in
\mathcal{T}$ and $(X,L)_{Y}$ is the subset of $\textrm{Hom}(X,L)$
consisting of morphisms $l: X\rightarrow L$ whose mapping cones
$\textrm{Cone}(l)$ are isomorphic to $Y$. Indeed, the structure
constant $\Phi_{XY}^L$ is related to counting triangles
$X\xrightarrow{f}L\xrightarrow{g}Y\xrightarrow{h}X[1]$. Note that
the Hall algebras defined by Ringel~\cite{Ringel1990} and derived
Hall algebras defined by To\"en~\cite{Toen2005} provided the most
canonical examples of Hall algebras as a tool of the
categorification. To define the derived Hall algebra of any
triangulated category under some (homological) finiteness conditions
by using To\"en formula becomes a very interesting question
(\cite[Remark 5.3]{Toen2005}). In~\cite{XX2006}, the authors
extended derived Hall algebras to Hall algebras associated to
triangulated categories $\mathcal{T}$ with some homological
finiteness conditions and the To\"en formula has the following
variant:
$$\Phi_{XY}^L=\prod_{i>0}|\Hom(X[i], Y)|^{(-1)^i}\cdot \sum_{\alpha\in V(X, Y; L)}\frac{|\mathrm{End}(X_1(\alpha))|}{|\mathrm{Aut}(X_1(\alpha))|},$$
where we refer to Proposition 2.1 in Section 2 for the definitions
of $V(X, Y; L)$ and $X_1(\alpha).$ In~\cite{SX}, the authors proved
that associated to the derived category $\md(\mathcal{A})$ of a
hereditary $k$-category $\mathcal{A}$ over a finite field $k$, the
derived Hall algebra $\mathcal{DH}(\mathcal{A})$ defined by To\"{e}n
can be identified with the lattice algebra $L(\mathcal{A})$ defined
by Kapranov [23] via the ``twist and extend'' procedure through a
suitable subalgebra closely related to the Heisenberg double.
However, none of these methods can yet be applied to the periodic
triangulated category (in particular, $2$-periodic triangulated
category (root category)), as it does not satisfy the homological
finiteness assumptions.

The purpose of this paper is to define a Hall algebra, associated to
a triangulated category $\mathcal{C}$ with the $t$-periodic functor
for any odd number $t>1$. We prove an analogue of the formula of
To\"{e}n \cite[Proposition 5.1]{Toen2005} and for any $X,Y,L\in
\mathcal{C}$, the structure constant
$$F_{XY}^L=\displaystyle\frac{|(X,L)_{Y}|}{|\mathrm{Aut}X|}\cdot
\big(\prod_{i=1}^t\displaystyle\frac{|\Hom(X[i],
L)|^{(-1)^i}}{|\Hom(X[i], X)|^{(-1)^i}}\big)^{\frac{1}{2}}.$$

As an application, we describe the Hall algebra associated to the
$3$-periodic orbit category of a hereditary abelian category. Note
that the $t$-periodic triangulated categories do not satisfy the
required homological finiteness conditions appeared
in~\cite{Toen2005,XX2006} and our results can be viewed as
non-trivial generalization of To\"{e}n's result \cite{Toen2005} and
Xiao-Xu's result~\cite{XX2006}. It will be of interest to deal with
those Lie algebras which arise from a Hall algebra over an odd
periodic triangulated category.

\section{Hall algebras arising from odd-periodic triangulated categories}
Let $k$ be a finite field with $q$ elements and $\mc$ a $k$-additive
triangulated category with the translation (or shift) functor
$T=[1]$ satisfying the following conditions:
\begin{enumerate}
\item[(1)] the homomorphism space $\Hom(X,Y)$ for any two
objects $X$ and $Y$ in $\mc$ is a finite dimensional $k$-space;
\item[(2)] the endomorphism ring $\ed X$ for any indecomposable
object $X$ in $\mc$ is a finite dimensional local $k$-algebra;
\item[(3)] $T^t=[1]^t=[t] \cong 1_{\mc}$ for some positive integer $t$.
\end{enumerate}
Then the category $\mc$ is called a $t$-periodic triangulated
category and $T=[1]$ is called a $t$-periodic translation (or shift)
functor. Note that the first two conditions imply the validity of
the Krull--Schmidt theorem in $\mc$, which means that any object in
$\mc$ can be uniquely decomposed into the direct sum of finitely
many indecomposable objects up to isomorphism.

Throughout this paper, we will assume that $t$ is an odd number and
$t>1.$ For any $X, Y$ and $Z$ in $\mc$, we will use $fg$ to denote
the composition of morphisms $f:X\rightarrow Y$ and $g:Y\rightarrow
Z,$ and $|S|$ to denote the cardinality of a finite set $S.$

Given $X,Y;L\in \mc,$ put
$$ W(X,Y;L)=\{(f,g,h)\in \Hom(X,L)\times
\Hom(L,Y)\times \Hom(Y,X[1])\mid$$$$
X\xrightarrow{f}L\xrightarrow{g}Y\xrightarrow{h}X[1] \mbox{ is a
triangle}\}.$$ \\
There is a natural action of $\aut X\times \aut Y$ on $W(X,Y;L).$
The orbit of $(f,g,h)\in W(X,Y;L)$ is denoted by
$$(f,g,h)^{\wedge}:=\{(af,gc^{-1},ch(a[1])^{-1})\mid (a,c)\in \aut X\times \aut Y\}.$$
The orbit space is denoted by $V(X,Y;L)=\{ (f,g,h)^{\wedge}|(f,g,h)
\in W(X,Y;L)\} .$ The radical of $\Hom(X, Y)$ is denoted by
$\mathrm{radHo}m(X, Y)$ which is the set $$\{f\in \Hom(X, Y)\mid gfh
\mbox{ is not an isomorphism }\mbox{for any }g: A\rightarrow X\mbox{
and }
$$
$$h:Y\rightarrow A\mbox{ with } A\in \mc\mbox{ indecomposable }\}. $$ Denote by $(X,Y)_Z$ the subset of $\Hom(X,Y)$
consisting of the morphisms whose mapping cones are isomorphic to
$Z.$

The Proposition $2.5$ in~\cite{XX2006} also holds for $t$--periodic
triangulated categories.
\begin{Prop}\label{mainproposition}
For any $Z,L,M\in \mc,$ we have the following:
\begin{enumerate}
    \item any $\alpha=(l,m,n)^{\wedge}\in V(Z,L;M)$ has the representative
of the form:
\begin{equation}\nonumber
\xymatrix{Z\ar[rr]^{\left(%
\begin{array}{c}
  0 \\
  l_2 \\
\end{array}%
\right)}&& M\ar[rr]^{\left(
                       \begin{array}{cc}
                         0 & m_2 \\
                       \end{array}
                     \right)
}&& L\ar[rr]^{\left(%
\begin{array}{cc}
  n_{11} & 0 \\
  0 & n_{22} \\
\end{array}%
\right)}&& Z[1]}
\end{equation}
where $Z=Z_1(\alpha)\oplus Z_2(\alpha),$ $L=L_1(\alpha)\oplus
L_2(\alpha)$, $n_{11}$ is an isomorphism between $L_1(\alpha)$ and
$Z_1(\alpha)[1]$ and $n_{22}\in
\mathrm{radHom}(L_2(\alpha),Z_2(\alpha)[1])$;
    \item $$
\frac{|(M,L)_{Z[1]}|}{|\mathrm{Aut}L|}=\sum_{\alpha\in
V(Z,L;M)}\frac{|\mathrm{End}
L_1(\alpha)|}{|n\Hom(Z[1],L)||\mathrm{Aut}L_1(\alpha)|}
$$
and
$$
\frac{|(Z,M)_{L}|}{|\mathrm{Aut}Z|}=\sum_{\alpha\in
V(Z,L;M)}\frac{|\mathrm{End}
Z_1(\alpha)|}{|\Hom(Z[1],L)n||\mathrm{Aut}Z_1(\alpha)|}
$$
where $ n\Hom(Z[1],L)=\{ b\in \mathrm{End}L \mid  b=ns\ \mbox{ for
some }\ s\in \Hom(Z[1],L)\} $ and $\Hom(Z[1],L)n=\{d\in
\mathrm{End}Z[1]\mid d=sn  \mbox{ for some }\ s\in \Hom(Z[1],L) \}$.
\end{enumerate}
\end{Prop}

\begin{lemma}\label{stablespace}
For any $(l,m,n)\in W(Z, L; M),$ we have
\begin{enumerate}
  \item $|n\Hom(Z[1], L)|=\big(\displaystyle\prod_{i=1}^t\displaystyle\frac{|\Hom(M[i], L)|^{(-1)^i}}
  {|\Hom(Z[i], L)|^{(-1)^i}|\Hom(L[i], L)|^{(-1)^i}}\big)^{\frac{1}{2}};$
  \item $|\Hom(Z[1], L)n|=\big(\displaystyle\prod_{i=1}^t\displaystyle\frac{|\Hom(Z[i], M)|^{(-1)^i}}
  {|\Hom(Z[i], L)|^{(-1)^i}|\Hom(Z[i], Z)|^{(-1)^i}}\big)^{\frac{1}{2}}.$
\end{enumerate}
\end{lemma}
\begin{proof}
We only prove the first identity, the proof for the second identity
is similar. By applying the functor $\Hom(-, L)$ to the triangle
$Z\xrightarrow{l}M\xrightarrow{m}L\xrightarrow{n}Z[1],$ we obtain a
long exact sequence
$$
\xymatrix{\cdots \ar[r]&\Hom(Z[t+1], L) \ar[r]^-{n[t]^*}&\Hom(L[t],
L)\ar[r]&&
\\
&\cdots\ar[r]&\Hom(Z[1], L)\ar[r]^{n^*}&\Hom(L, L)\ar[r]&\cdots}
$$
Since $n\Hom (Z[1], L)=$ Image of $n^*$, $n[t]\Hom(Z[t+1], L)=$
Image of $n[t]^*$ and $[t] \cong 1_{\mc}$, the above long exact
sequence induces the following long exact sequence
$$
\xymatrix{0\ar[r]&n\Hom(Z[1], L)\ar[r]&\cdots\ar[r]&n\Hom(Z[1],
L)\ar[r]&0}
$$
which implies the desired identity.

\end{proof}

By Proposition~\ref{mainproposition} and Lemma~\ref{stablespace}, we
have the following immediate result.
\begin{Cor}\label{symmetry}
With the above notations, we have
$$\frac{|(M,L)_{Z[1]}|}{|\mathrm{Aut}L|}\cdot
\big(\prod_{i=1}^t\frac{|\Hom(M[i], L)|^{(-1)^i}}{|\Hom(L[i],
L)|^{(-1)^i}}\big)^{\frac{1}{2}}=\frac{|(Z,M)_{L}|}{|\mathrm{Aut}Z|}\cdot
\big(\prod_{i=1}^t\frac{|\Hom(Z[i], M)|^{(-1)^i}}{|\Hom(Z[i],
Z)|^{(-1)^i}}\big)^{\frac{1}{2}}.$$
\end{Cor}

Recall some notations appeared in~\cite{XX2006}. Let $X,Y,Z, L, L'$
and $M$ be in $\mc$. Define
$$
\Hom(M\oplus X,L)^{Y,Z[1]}_{L'[1]}:=\{\left(%
\begin{array}{c}
  m \\
  f \\
\end{array}%
\right)\in \Hom (M\oplus X,L)\mid$$$$\hspace{3cm}
\mathrm{Cone}(f)\simeq Y, \mathrm{Cone}(m)\simeq Z[1] \mbox{ and }\mathrm{Cone}\left(%
\begin{array}{c}
  m \\
  f \\
\end{array}%
\right)\simeq L'[1]\}
$$
and
$$
\Hom(L',M\oplus X)^{Y,Z[1]}_{L}:=\{(f',-m')\in \Hom(L',M\oplus
X)\mid
$$$$\hspace{3cm}\mathrm{Cone}(f')\simeq Y, \mathrm{Cone}(m')\simeq Z[1]
\mbox{ and } \mathrm{Cone}(f',-m')\simeq L\}.
$$
Note that the orbit space of $\Hom(M\oplus X,L)^{Y,Z[1]}_{L'[1]}$
under the action of $\aut L$ is just the orbit space of
$\Hom(L',M\oplus X)^{Y,Z[1]}_{L}$  under the action of $\aut L'$
(see \cite{XX2006}), which is denoted by $V(L',L;M\oplus
X)_{Y,Z[1]}$.

If in Corollary \ref{symmetry}, we substitute $M$ by $M\oplus X$ and
 $Z$ by $L'$, then we can have the following result as an analog of
Corollary \ref{symmetry} by using a similar proof.
\begin{Cor}\label{symmetry2}
With the above notations, we have
$$\frac{|\Hom(M\oplus X,L)^{Y, Z[1]}_{L'[1]}|}{|\mathrm{Aut}L|}\cdot
\big(\prod_{i=1}^t\frac{|\Hom((M\oplus X)[i],
L)|^{(-1)^i}}{|\Hom(L[i],
L)|^{(-1)^i}}\big)^{\frac{1}{2}}$$$$=\frac{|\Hom(L',M\oplus
X)^{Y,Z[1]}_{L}|}{|\mathrm{Aut}L'|}\cdot
\big(\prod_{i=1}^t\frac{|\Hom(Z[i], M)|^{(-1)^i}}{|\Hom(Z[i],
Z)|^{(-1)^i}}\big)^{\frac{1}{2}}.$$
\end{Cor}


For any $Z, L$ and $M\in \mc$, set
\begin{displaymath}
\begin{array}{lrl}
F_{ZL}^M
&:=&\displaystyle\frac{|(M,L)_{Z[1]}|}{|\mathrm{Aut}L|}\cdot
\big(\prod_{i=1}^t \displaystyle\frac{|\Hom(M[i],
L)|^{(-1)^i}}{|\Hom(L[i],
L)|^{(-1)^i}}\big)^{\frac{1}{2}}\\&=&\displaystyle\frac{|(Z,M)_{L}|}{|\mathrm{Aut}Z|}\cdot
\big(\prod_{i=1}^t\displaystyle\frac{|\Hom(Z[i],
M)|^{(-1)^i}}{|\Hom(Z[i], Z)|^{(-1)^i}}\big)^{\frac{1}{2}}.
\end{array}
\end{displaymath}
This formula is an analogue of To\"en's formula~\cite[Proposition
5.1]{Toen2005}. Set $q=v^2$, then $F_{ZL}^M\in \bbq[v, v^{-1}].$ Let
$\bbq(v, v^{-1})$ be the rational field of $\bbq[v, v^{-1}].$ For
any $X \in \mc$, we denote its isomorphism class by $[X]$.

Now we can define an associative algebra arising from $\mc$ using
$F_{ZL}^M$ as the structure constant.

\begin{theorem}\label{maintheorem123}
Let $\mathcal{H}(\mc)$ be the $\bbq(v, v^{-1})$-space with the basis
$\{u_{[X]}\mid X\in \mc\}$. Endowed with the multiplication defined
by
$$
u_{[X]}*u_{[Y]}=\sum_{[L]}F_{XY}^L u_{[L]},$$ $\mathcal{H}(\mc)$ is
an associative algebra with the unit  $u_{[0]}.$

\end{theorem}
\begin{proof}
In order to simplify the notation, for $X,Y\in\mc,$ we set
$$\{X,Y\}=\prod_{i=1}^t|\Hom(X[i],Y)|^{(-1)^{i}}.$$ For $X,Y,Z \mbox{ and }M\in\mc,$ we will
prove that $u_Z*(u_X*u_Y)=(u_Z*u_X)*u_Y.$ It is equivalent to prove
$$
\sum_{[L]}F_{XY}^{L}F_{ZL}^{M}=\sum_{[L']}F_{ZX}^{L'}F_{L'Y}^{M}.
$$We know that
$$
\sum_{[L]}F_{XY}^{L}F_{ZL}^{M}=\sum_{[L]}\frac{|(X,L)_Y|}{|\mathrm{Aut}(X)|\{X,X\}^{\frac{1}{2}}}\{X,L\}^{\frac{1}{2}}\cdot
\frac{|(M,L)_{Z[1]}|}{|\mathrm{Aut}(L)|\{L,L\}}\{M,L\}^{\frac{1}{2}}
$$
By Proposition $3.5$ in~\cite{XX2006},
$\displaystyle\sum_{[L]}F_{XY}^{L}F_{ZL}^{M}$ is equal to
$$
\frac{1}{|\aut
X|\cdot\{X,X\}^{\frac{1}{2}}}\cdot\sum_{[L]}\sum_{[L']}\frac{|\Hom(M\oplus
X,L)^{Y,Z[1]}_{L'[1]}|}{|\aut L|}\cdot \frac{\{M\oplus
X,L\}^{\frac{1}{2}}}{\{L,L\}^{\frac{1}{2}}}.
$$
Dually, $\displaystyle\sum_{[L']}F_{ZX}^{L'}F_{L'Y}^{M}$  is equal
to
$$
\frac{1}{|\aut
X|\cdot\{X,X\}^{\frac{1}{2}}}\sum_{[L']}\cdot\sum_{[L]}\frac{|\Hom(L',M\oplus
X)^{Y,Z[1]}_{L}|}{|\aut L'|}\cdot \frac{\{L',M\oplus
X\}^{\frac{1}{2}}}{\{L',L'\}^{\frac{1}{2}}}.
$$
By Corollary \ref{symmetry2}, the proof of the theorem follows
immediately.
\end{proof}

\section{Hall algebras for 3-periodic relative derived categories of hereditary categories}
In this section, first of all, we will describe the relative derived
categories of $t$-cycle complexes over hereditary abelian
categories, which are the concrete examples of $t$--periodic
triangulated categories. Then, under the case when $t=3$, we will
provide a detailed description of the Hall algebra over a
$3$--periodic triangulated category.
\subsection{$t$-cycle complexes and relative homotopy categories}

Let $\mathcal{A}$ be an abelian category. Let $t\in \bbn$ and $t>1.$
Recall from~\cite{PX1997} that a $t$--cycle complex over
$\mathcal{A}$ is a complex $\dot{X}_t=(X^i, d_X^i )_{i\in
\mathbb{Z}}$ such that $X^i=X^j$  and $d_X^i=d_X^j$ for all $i, j
\in \mathbb{Z}$ with $i \equiv j (\mood\, t)$, where all $X^i$ are
objects in $\mathcal{A}$ and all $d_X^i: X^i \longrightarrow
X^{i+1}$ are morphisms in $\mathcal{A}$ with $ d_X^i d_X^{i+1}=0$.
For convenience, we can express $\dot{X}_t$ by using the form:
$$
\xymatrix{X_1\ar[r]^{d_X^1}&X_2\ar[r]^{d_X^2}&\cdots\ar[r]&X_t\ar[r]^{d_X^t}&X_1}.
$$  If
$\dot{X}_t$ and $\dot{Y}_t$ are two $t$--cycle complexes, a morphism
$\dot{f}_t: \dot{X}_t \longrightarrow \dot{Y}_t$ is a sequence of
morphisms $f^i: X^i \longrightarrow Y^i$ in $\mathcal{A}$ such that
$f^i=f^j$ for all $i, j \in \mathbb{Z}$ with $i \equiv j (\mood\,
t)$ and $d_X^i f^{i+1}=f^id_Y^i$ for all $i \in \mathbb{Z}$. These
morphisms are composed in an obvious way. So all $t$--cycle
complexes with these morphisms constitute an abelian category
$\mathcal{C}_t (\mathcal{A})$ which is called the $t$--cycle complex
category.

Let $\dot{f}_t, \dot{g}_t: \dot{X}_t \longrightarrow \dot{Y}_t$ be
two morphisms of $t$--cycle complexes. A relative homotopy
$\dot{s}_t$ from $\dot{f}_t$ to $\dot{g}_t$ a sequence of morphisms
$s^i: X^i \longrightarrow Y^{i-1}$ in $\mathcal{A}$ such that
$s^i=s^j$ for all $i, j \in \mathbb{Z}$ with $i \equiv j (\mood
\,t)$ and $f^i-g^i=s^i d_Y^{i-1}+d_X^i s^{i+1}$ for all $i\in
\mathbb{Z}$. Morphisms $\dot{f}_t$ and $\dot{g}_t$ are said to be
relatively homotopic if there exists a relative homotopy from
$\dot{f}_t$ to $\dot{g}_t$. Relative homotopy is an equivalence
relation compatible with composition of morphisms. So we can form a
new additive category $\mathcal{K}_t (\mathcal{A})$, called the
relative homotopy category of $t$--cycle complexes over
$\mathcal{A}$, by considering all $t$--cycle complexes as objects
and the additive group of relative homotopy classes of morphisms
from $\dot{X}_t$ to $\dot{Y}_t$ in $\mathcal{C}_t(\mathcal{A})$ as
the group of morphisms from $\dot{X}_t$ to $\dot{Y}_t$ in
$\mathcal{K}_t(\mathcal{A})$. As in usual complex categories, one
can define quasi-isomorphisms in $\mathcal{C}_t(\mathcal{A})$ and
$\mathcal{K}_t(\mathcal{A})$. Localizing $\mathcal{K}_t
(\mathcal{A})$ with respect to the set of all quasi-isomorphisms,
one can get an additive category, denoted by $\mathcal{D}_t
(\mathcal{A})$, called the relative derived category of $t$--cycle
complexes over $\mathcal{A}$ (see \cite[Section 3.2]{LP}). As in the
appendix of \cite{PX1997}, $\mathcal{K}_t(\mathcal{A})$ is a
triangulated category with the translation functor $[1]$ which is
the stable category of a Frobenius category. In the same way,
$\md_t(\mathcal{A})$ is also a triangulated category with the
translation functor $[1]$. As usual we denote by
$\mathcal{C}^b(\mathcal{A})$ the  category of bounded complexes over
$\mathcal{A}$, and by $\mathcal{K}^b(\mathcal{A})$  the homotopy
category of bounded complexes over $\mathcal{A}$. Let
$\mathcal{D}^b(\mathcal{A})$ be the derived category of
$\mathcal{A}$ obtained from $\mathcal{K}^b(\mathcal{A})$  by
localizing with respect to the set of quasi--isomorphisms. Then
$\mathcal{K}^b(\mathcal{A})$ and $\mathcal{D}^b(\mathcal{A})$ are
all triangulated categories with the translation functor--the shift
functor $[1]$.

As in \cite{PX1997}, we can define a functor $F:
\mathcal{C}^b(\mathcal{A}) \longrightarrow
\mathcal{C}_t(\mathcal{A})$ as follows. For $\dot{X}=(X^i, d_X^i)
\in \mathcal{C}^b(\mathcal{A})$, set $F\dot{X}=((F\dot{X})^i,
d_{F\dot{X}}^i)$ where $F\dot{X}=\oplus_{m\in \mathbb{Z}}X^{i+mt}
\text{ and } d_{F\dot{X}}^i=(d_{sm}^i)_{s,m\in \mathbb{Z}}$ such
that $d_{sm}^i: X^{i+st}\longrightarrow X^{(i+1)+mt}$ with
$d_{sm}^i=0$ for $s\neq m$ and $d_{ss}^i=d_X^i$ for all $s\in
\mathbb{Z}$. For $\dot{f}=(f^i):\dot{X} \longrightarrow \dot{Y}$ in
$\mathcal{C}^b(\mathcal{A})$, set $F(\dot{f})=(g^i)_{i\in
\mathbb{Z}}: F\dot{X} \longrightarrow F\dot{Y}$ where
$g^i=(g_{sm}^i)_{s,m\in\mathbb{Z}}$ such that $g_{sm}^i: X^{i+st}
\longrightarrow Y^{i+mt}$ with $g_{sm}^i=0$ for $s\neq m$ and
$g_{ss}^i=f^i$ for all $s\in \mathbb{Z}$. It is not difficult to
check that $F$ is a well--defined functor.

It is not hard to verify that $F$ is a Galois covering with the
Galois group $G=\langle[1]^t\rangle$ and $F$ is exact, i.e. $F$
sends a chainwise split exact sequence to a chainwise split exact
sequence. In \cite{PX1997}, the authors proved that $F$ induces a
Galois covering from $\mathcal{K}^b(\mathcal{A})$ to
$\mathcal{K}_t(\mathcal{A})$ with the Galois group
$\langle[1]^t\rangle$ and it is exact, i.e. it sends a triangle to a
triangle. In the same way, $F$ induces a Galois covering from
$\mathcal{D}^b(\mathcal{A})$ to $\mathcal{D}_t(\mathcal{A})$,
denoted still by $F.$ For any $X, Y\in \mathcal{A}$, let
$\tilde{X}:=F(X)$ and $\tilde{Y}:=F(Y)$, we have
\begin{equation}\label{isomorphism}
\Hom_{\md_t(\mathcal{A})}(\tilde{X}, \tilde{Y})\cong
\bigoplus_{F(X')=\tilde{X}}\Hom(X', Y)\cong
\bigoplus_{F(Y')=\tilde{Y}}\Hom(X, Y').
\end{equation}

The following proposition is similar to Proposition $3.2$
in~\cite{LP}.
\begin{Prop}\label{indecomposable}
Let $\mathcal{A}$ be a hereditary (abelian) category, i.e.,
$\mathrm{Ext}^i(-, -)=0$ for $i>1$ and $F:
\md^b(\mathcal{A})\rightarrow \md_t(\mathcal{A})$ be the Galois
covering functor for $t\in\bbn$ and $t>1$. Then $F$ is dense.
\end{Prop}
\begin{proof}
Given a $t$-cycle complex $\dot{X}$ with the form as follows:
$$\xymatrix{X_1\ar[r]^{f_1}&X_2\ar[r]^{f_2}&X_3\ar[r]^{f_3}&\cdots\ar[r]&X_{t}\ar[r]^{f_t}&X_1}.$$
For $f_1: X_1\rightarrow X_2$, there exists $L\in \mathcal{A}$ such
that the following diagram
$$
\xymatrix{&\mathrm{Ker}f_1\ar@{=}[r]\ar[d]^{i_1}&\mathrm{Ker}f_1\ar[d]^{i_2}&&\\
0\ar[r]&X_1\ar[r]^{g_1}\ar[d]^{h_1}&L\ar[r]\ar[d]^{h_2}&\mathrm{Coker}f_1\ar[r]\ar@{=}[d]&0\\
0\ar[r]&\mathrm{Im}f_1\ar[r]^{g_2}&X_2\ar[r]&\mathrm{Coker}f_1\ar[r]&0}
$$
is commutative.

This diagram induces a short exact sequence in $\mathcal{C}_t
(\mathcal{A})$:
$$
\xymatrix@C=0.7cm{0\ar[r]\ar[d]&0\ar[r]\ar[d]&0\ar[r]\ar[d]&\cdots\ar[r]&0\ar[r]\ar[d]&0\ar[d]\\
\mathrm{Ker}f_1\ar@{=}[r]\ar[d]_{\left(
                                             \begin{array}{cc}
                                               -i_1 & 1 \\
                                             \end{array}
                                           \right)
}&\mathrm{Ker}f_1\ar[d]^{i_2}\ar[r]&0\ar[r]\ar[d]&\cdots\ar[r]&0\ar[r]\ar[d]&\mathrm{Ker}f_1\ar[d]^{\left(
                                             \begin{array}{cc}
                                               -i_1 & 1 \\
                                             \end{array}
                                           \right)
}\\
X_1\oplus \mathrm{Ker}f_1\ar[r]^-{\left(
                                             \begin{array}{c}
                                               g_1 \\
                                               0 \\
                                             \end{array}
                                           \right)
}\ar[d]_{\left(
                                                         \begin{array}{c}
                                                           1 \\
                                                          i_1 \\
                                                         \end{array}
                                                       \right)
}&L\ar[r]^{f_2}\ar[d]^{h_2}&X_3\ar[r]^{f_3}\ar@{=}[d]&\cdots\ar[r]&X_t\ar[r]^-{
                                                                                  (0\quad f_t)    }\ar@{=}[d]&X_1\oplus
\mathrm{Ker}f_1\ar[d]^{\left(
                                                         \begin{array}{c}
                                                           1 \\
                                                          i_1 \\
                                                         \end{array}
                                                       \right)
}\\
X_1\ar[r]^{f_1}\ar[d]&X_2\ar[r]^{f_2}\ar[d]&X_3\ar[r]^{f_3}\ar[d]&\cdots\ar[r]&X_t\ar[r]^{f_t}\ar[d]&X_1\ar[d]\\
0\ar[r]&0\ar[r]&0\ar[r]&\cdots\ar[r]&0\ar[r]&0}
$$
Note that since $\mathrm{Im}f_t\subseteq \mathrm{Ker}f_1$, we can
write $f_t: X_t\rightarrow \mathrm{Ker}f_1$ without causing any
confusion. Since the $t$--cycle complex on line $2$ is contractible,
two $t$--cycle complexes on line $3$ and $4$ are quasi-isomorphic to
each other. Let $X^{\bullet}\in \md^b(\mathcal{A})$ be the complex
$$
\xymatrix{\cdots
\ar[r]&0\ar[r]&X_1\ar[r]^{f_1}&\cdots\ar[r]&X_t\ar[r]^-{f_t}&\mathrm{Ker}f_1\ar[r]&0\ar[r]&\cdots}.
$$
Then $F(X^{\bullet})$ is just the $t$--cycle complex on line $3$ and
then isomorphic to $\dot{X}$ in $\md_t(\mathcal{A}).$ Hence, $F$ is
dense.
\end{proof}

\begin{remark} Now let $\Lambda$ be a basic finite dimensional associative
algebra with unity over a field. Let $\textrm{mod}\, \Lambda$ be the
category of all finitely generated (right) $\Lambda$--modules and
$\mathcal{P}$ the full subcategory of $\textrm{mod}\, \Lambda$ with
projective $\Lambda$-modules as objects. For simplicity, we write
$\mathcal{D}^b(\Lambda)$ instead of $\mathcal{D}^b(\textrm{mod}\,
\Lambda)$ and call it the derived category of $\Lambda$. Let
$\mathcal{K}_t(\mathcal{P})$ be the relative homotopy category of
$t$-cycle complex over $\mathcal{P}$. If $\Lambda$ is hereditary,
then $\mathcal{D}^b(\Lambda)$  is a Galois cover of
$\mathcal{K}_t(\mathcal{P})$ with the Galois group
$\langle[1]^t\rangle$ and the covering functor is exact and dense,
where $[1]$ is the translation functor of the triangulated category
$\mathcal{D}^b(\Lambda)$. Under the covering functor
$F':\mathcal{D}^b(\Lambda) \longrightarrow
\mathcal{K}_t(\mathcal{P})$, $\textrm{mod}\, \Lambda$ can be fully
embedded in $\mathcal{K}_t(\mathcal{P})$. This full embedding may be
obtained directly by sending $\Lambda$--modules to the $t$--cycle
complexes over $\mathcal{P}$ which is naturally induced by their
projective resolutions. Similarly with Happel's discussions
in~\cite{Happel}, one can prove that $\mathcal{K}_t(\mathcal{P})$
has Auslander--Reiten triangles. Specially in case $t=2$, the orbit
category $\mathcal{D}^b(\Lambda)/[1]^2$ is called the root category
of $\Lambda$ by Happel. Then $\mathcal{D}^b(\Lambda)/[1]^2
\backsimeq \mathcal{K}_2 (\mathcal{P})$ is a triangulated category
and the covering $F': \mathcal{D}^b(\Lambda)\longrightarrow
\mathcal{D}^b(\Lambda)/[1]^2$ is exact. Therefore,
$\mathcal{D}^b(\Lambda)/[1]^2 \simeq \md_2 (\textrm{mod}\,\Lambda)$.
\end{remark}

There is a full embedding of $\mathcal{A}$ into $\md_t
(\mathcal{A})$ which sends each object $X$ of $\mathcal{A}$ into the
stalk complex $ \xymatrix{
 0\ar[r]^-{0}&  \cdots \ar[r]^{0} &0\ar[r]^{0}    & X\ar[r]^{0}&0}$ be a
 $t$--cycle complex. We will identify this complex with $X$.

Let $\mathrm{ind}\mathcal{A}$ be the set of isomorphism classes of
indecomposable objects in $\mathcal{A}.$ Proposition
\ref{indecomposable} implies that if $\mathcal{A}$ is hereditary,
then the set of isomorphism classes of indecomposable objects in
$\md_t(\mathcal{A})$ is
$$
\{X[i]\mid i=0,1,\cdots,t-1, X \in \mathrm{ind}\mathcal{A}\}.
$$

In the following, we will only deal with the case when $t=3$ for
simplicity.

\begin{Prop}\label{indecomposable2}Let
$ \xymatrix{\dot{X}:\,\,
X_1\ar[r]^-{f_1}&X_2\ar[r]^{f_2}&X_3\ar[r]^{f_3}&X_1}$ be a
 $3$--cycle complex over a hereditary abelian category $\mathcal{A}$. Then in $\md_3(\mathcal{A}),$ $\dot{X}$ is isomorphic to
 $$
\xymatrix{\mathrm{Ker}f_1\ar[r]^-{0}&\mathrm{Ker}f_2/\mathrm{Im}f_1\ar[r]^-{0}&\mathrm{Coker}f_2\ar[r]^-{f}&\mathrm{Ker}f_1}
 $$
where $f$ is naturally induced by $f_3.$
\end{Prop}
\begin{proof}
In $\mathcal{C}_3(\mathcal{A})$, there exists a natural short exact
sequence involving $\dot{X}$ as follows:
\begin{equation}\label{ses1} \xymatrix{
&0\ar[r]\ar[d]&0\ar[r]\ar[d]&0\ar[r]\ar[d]&0\ar[d]\\
&\mathrm{Ker}f_1\ar[r]\ar[d]&0\ar[r]\ar[d]&0\ar[r]\ar[d]&\mathrm{Ker}f_1\ar[d]\\
&X_1\ar[r]^{f_1}\ar[d]&X_2\ar[r]^{f_2}\ar[d]&X_3\ar[r]^{f_3}\ar[d]&X_1\ar[d]\\
\dot{N}:&\mathrm{Im}f_1\ar[r]^{i}\ar[d]&X_2\ar[r]^{f_2}\ar[d]&X_3\ar[r]^{\pi}\ar[d]&\mathrm{Im}f_1\ar[d]\\
&0\ar[r]&0\ar[r]&0\ar[r]&0}
\end{equation}
where $i$ is the natural embedding and $\pi$ is induced by $f_3.$
The complex $\dot{N}$ in above diagram induces the following short
exact sequence in $\mathcal{C}_3(\mathcal{A})$:
$$
\xymatrix{&0\ar[r]\ar[d]&0\ar[r]\ar[d]&0\ar[r]\ar[d]&0\ar[d]\\
&\mathrm{Im}f_1\ar[r]\ar[d]&\mathrm{Im}f_1\ar[r]\ar[d]&0\ar[r]\ar[d]&\mathrm{Im}f_1\ar[d]\\
\dot{N}:&\mathrm{Im}f_1\ar[r]^{i}\ar[d]&X_2\ar[r]^{f_2}\ar[d]&X_3\ar[r]^{\pi}\ar[d]&\mathrm{Im}f_1\ar[d]\\
\dot{L}:&0\ar[r]\ar[d]&\mathrm{Coker}f_1 \ar[d]\ar[r]^{g} \ar[d]&X_3\ar[r]\ar[d]&0 \ar[d]\\
&0\ar[r]&0\ar[r]&0\ar[r]&0}
$$
where $g$ is induced by $f_2.$ Hence, $\dot{N}$ is isomorphic to the
bottom complex (denoted by $\dot{L}$) in $\md_3(\mathcal{A})$. Since
the complex $\xymatrix{0\ar[r]&\mathrm{Im}f_2\ar[r]&X_3\ar[r]&0}$ is
isomorphic to $\mathrm{Coker}f_2[1]$ in $\md_3(\mathcal{A})$, we
obtain a triangle in $\md_3(\mathcal{A})$
$$
\xymatrix{\mathrm{Ker}f_2/\mathrm{Im}f_1[2]\ar[r]&\dot{L}\ar[r]&\mathrm{Coker}f_2[1]\ar[r]^{h}&\mathrm{Ker}f_2/\mathrm{Im}f_1}.
$$
By using isomorphisms in~\eqref{isomorphism}, we have
$$
\Hom_{\md_3(\mathcal{A})}(\mathrm{Coker}f_2[1],
\mathrm{Ker}f_2/\mathrm{Im}f_1)=0.
$$
Hence, $h=0$ and then the triangle is split. So $\dot{N}$ is
isomorphic to $\mathrm{Ker}f_2/\mathrm{Im}f_1[2]\oplus
\mathrm{Coker}f_2[1].$ By using the isomorphisms
in~\eqref{isomorphism} again, we deduce that the short exact
sequence \eqref{ses1} induces the triangle
$$
\xymatrix{\mathrm{Ker}f_1\ar[r]&\dot{X}\ar[r]&\mathrm{Ker}f_2/\mathrm{Im}f_1[2]\oplus
\mathrm{Coker}f_2[1]\ar[r]^-{t}&\mathrm{Ker}f_1[1]}
$$
where $t=\left(
           \begin{array}{c}
             0 \\
             f[1] \\
           \end{array}
         \right)
$ and $f$ is induced by $f_3.$ The proposition is proved.
\end{proof}
By using Proposition~\ref{indecomposable2}, we can easily decompose
any object in $\md_3(\mathcal{A})$ into the direct sum of
indecomposable objects. Indeed, for the given $3$-cycle complex $
\xymatrix{\dot{X}:\,\,
 X_1\ar[r]^-{f_1}&X_2\ar[r]^{f_2}&X_3\ar[r]^{f_3}&X_1}$,
if $X_2=0$, then in $\md_3(\mathcal{A})$, $\dot{X}$ is isomorphic to
$\xymatrix{0\ar[r]&\mathrm{Ker}f_2\ar[r]^{0}&\mathrm{Coker}f_2\ar[r]&0}$.
This implies that in $\md_3(\mathcal{A})$, the complex $\dot{X}$ in
Proposition \ref{indecomposable2} is isomorphic to
$$
\mathrm{Ker}f_2/\mathrm{Im}f_1[2]\oplus \mathrm{Ker}f[1]\oplus
\mathrm{Coker}f.
$$

Now let $\mathcal{A}$ be the category of all finitely generated
modules of a basic finite dimensional associative hereditary algebra
over a finite field $k$. Let $U$ be the associative and unital
$\bbq(v, v^{-1})$-algebra generated by the set
$$
\{u^{[i]}_{[X]}\mid i=0,1,2,  X\in \mathcal{A}\}
$$
with the following generating relations:
\begin{enumerate}
  \item for $n=0,1,2,$ $u_{[X]}^{[n]}\cdot u_{[Y]}^{[n]}=\displaystyle\sum_{[L]}F_{XY}^L\cdot
  u_{[L]}^{[n]}$; {\vspace{0.2cm}}
  \item for $n=0, 1,$ $u_{[X]}^{[n]}\cdot u_{[Y]}^{[n+1]}=\displaystyle\sum_{[K],[C]}F_{X,Y[1]}^{K[1]\oplus C}\cdot
  u_{[K]}^{[n+1]}\cdot u_{[C]}^{[n]}$; {\vspace{0.2cm}}
  \item $u_{[X]}^{[2]}\cdot u_{[Y]}^{[0]}=\displaystyle\sum_{[K],[C]}F_{X,Y[1]}^{K[1]\oplus C}\cdot
  u_{[K]}^{[0]}\cdot u_{[C]}^{[2]}$.
\end{enumerate}
\begin{theorem}\label{structure}
 Let $\mathcal{H}(\md_3(\mathcal{A}))$ be the
$\bbq(v, v^{-1})$-algebra defined in Theorem \ref{maintheorem123}.
Then there is an isomorphism of algebras from $U$ to
$\mathcal{H}(\md_3(\mathcal{A}))$.
\end{theorem}
\begin{proof}
There is a natural linear map of $\bbq(v, v^{-1})$-spaces
$$
\Phi:U\rightarrow \mathcal{H}(\md_3(\mathcal{A}))
$$
defined by $\Phi(u^{[i]}_{[X]})=u_{[X[i]]}.$ It is clear that the
relations $(1), (2)$ and (3) in the Definition of $U$ are satisfied
in $\mathcal{H}(\md_3(\mathcal{A}))$. Hence, $\Phi$ is a
homomorphism of algebras. Now we define a partial order over the set
of isomorphism classes of objects in $\md_3(\mathcal{A})$. Any
object in $\md_3(\mathcal{A})$ is isomorphic to $X\oplus Y[1]\oplus
Z[2]$ for some objects $X, Y$ and $Z$ in $\mathcal{A}$. We say
$$[X_1\oplus Y_1[1]\oplus Z_1[2]]< [X_2\oplus Y_2[1]\oplus Z_2[2]]$$ if
$\udim X_1\leq \udim X_2, Y_1\leq \udim Y_2$ and $\udim X_3\leq
\udim Y_3$ and at least one is not the equation. We claim that
$\Phi$ is surjective. For any $X, Y$ and $Z$ in $\mathcal{A}$, we
have
\begin{equation}\label{expression2}
u_{[X\oplus Y[1]\oplus Z[2]]}=u_{Z[2]}\cdot u_{Y[1]}\cdot u_X
-\sum_{[K],[C]; K\neq X, C\neq Z}F_{Z, X[1]}^{K[1]\oplus
C}u_{[K\oplus Y[1]\oplus C[2]]}.
\end{equation}
It is clear that $[K\oplus Y[1]\oplus C[2]]< [X\oplus Y[1]\oplus
Z[2]]$ for any $K, C$ occurred  in the right hand of the above
equation since $K\neq Z$ and $C\neq Y$ imply that $\udim K<\udim Z$
and $\udim C<\udim Y.$ We use the induction on the partial order and
deduce that there exists $u\in U$ such that $\Phi(u)=u_{[X\oplus
Y[1]\oplus Z[2]]}.$ Hence, $\Phi$ is surjective. Next, we prove that
$\Phi$ is injective. For any $\mu\in U,$ it can be reformulated to
have the following form
\begin{equation}\label{expression}
\mu=\sum_{[X], [Y], [Z]}a_{XYZ}\cdot u_{[Z]}^{[2]}\cdot
u_{[Y]}^{[1]}\cdot u_{[X]}^{[0]}.
\end{equation}
where $a_{XYZ}\in \bbq(v, v^{-1})$. Indeed, it is enough to check
the case when $\mu$ is equal to $u_{[X]}^{[0]}\cdot u_{[Z]}^{[2]}$,
$u_{[X]}^{0}\cdot u_{[Y]}^{[1]}$ or $u_{[Y]}^{[1]}\cdot
u_{[Z]}^{[2]}$. By following the relation $(3)$ in the Definition of
$U$, we have
$$
u_{[X]}^{[0]}\cdot u_{[Z]}^{[2]}=u_{[Z]}^{[2]}\cdot
u_{[X]}^{[0]}-\sum_{[K], [C]; [K\oplus C[2]]<[X\oplus
Z[2]]}u_{[K]}^{[0]}\cdot u_{[C]}^{[2]}.
$$
By the induction on the partial order, we can rewrite
$u_{[X]}^{[0]}\cdot u_{[Z]}^{[2]}$ with the form~\eqref{expression}.
The discussions of $u_{[X]}^{0}\cdot u_{[Y]}^{[1]}$ and
$u_{[Y]}^{[1]}\cdot u_{[Z]}^{[2]}$ are similar. By using the
expression~\eqref{expression}, we assume that
$$
\Phi(\mu)=\sum_{[X], [Y], [Z]}a_{XYZ}\cdot u_{[Z[2]]}\cdot
u_{[Y[1]]}\cdot u_{[X]}=0.
$$
However, by using the expression \eqref{expression2}, we have
\begin{displaymath}
\begin{array}{lcl}
& &u_{[Z[2]]}\cdot u_{[Y[1]]}\cdot u_{[X]} \\
&=& u_{[X\oplus Y[1]\oplus Z[2]]}+\displaystyle\sum_{[K], [C];
[K\oplus Y[1]\oplus C[2]]<[X\oplus Y[1]\oplus Z[2]]}F_{Z,
X[1]}^{K[1]\oplus C}u_{[K\oplus Y[1]\oplus C[2]]}.
\end{array}
\end{displaymath}
Hence, we obtain
$$
\sum_{[X], [Y], [Z]}a_{XYZ}\cdot u_{[X\oplus Y[1]\oplus Z[2]]}=0.
$$
This implies $a_{XYZ}=0$ and then $\mu=0.$ The proof of the theorem
is completed.
\end{proof}

\begin{remark}
Since $X=X[3]$ for any $X\in \md_3(\mathcal{A})$, the Grothendieck
group of $\md_3(\mathcal{A})$ is $\mathbb{Z}_2$-graded.  It will be
of interest to deal with the $\mathbb{Z}_2$-graded Lie algebras
which arise from a Hall algebra over a $3$-periodic triangulated
category.

\end{remark}
\begin{center}
{\textbf{Acknowledgements.}}
\end{center}
Fan Xu is grateful to Claus Michael Ringel for his help and
hospitality while staying at University of Bielefeld as an Alexander
von Humboldt Foundation fellow.

\bibliographystyle{amsalpha}

\begin{thebibliography}{10}



\bibitem{CK2008}
P. Caldero, B. Keller, \emph{From triangulated categories to cluster
algebras}, Invent. Math. \textbf{172} (2008), no. 1, 169--211.


\bibitem{Green}
J. A. Green, \textit{Hall algebras, hereditary algebras and quantum
groups}, Invent.Math.\textbf{120} (1995), 361-377.

\bibitem{Hall} P.~Hall, \emph{The algebra of partitions}, in: Proceedings 4th Canadian Mathematical Congress, Banff, 1957, University
of Toronto Press, 1958, 147--159.


\bibitem{Happel} D.~Happel, \emph{On the derived category of a finite-dimensional
algebra}, Comment. Math. Helv. \textbf{62} (1987), 339--389.



\bibitem{Joyce} D.~Joyce,  \textit{Configurations in abelian categories. II. Ringel-Hall algebras}, Advances in Mathematics {\bf 210} (2007), 635-706.
\bibitem{Kashiwara} M.~Kashiwara, \emph{On crystal bases of the $q$-analogue of universal
enveloping algebras}, Duke Math.~J., 63(1991), no.~2, 465--516.


\bibitem{Kap1}
M. Kapranov, \emph{Eisenstein series and quantum affine algebras},
Algebraic geometry, 7. J. Math. Sci. (New York) \textbf{84}, (1997)
no. 5, 1311--1360.

\bibitem{Kap2}
M. Kapranov, \emph{ Heisenberg doubles and derived categories}, J.
Algebra \textbf{202}, (1998) no. 2, 712--744.

\bibitem{LP} Y.~Lin, L.~Peng, \emph{Elliptic Lie algebras and tubular algebras},
Adv. Math. \textbf{196} (2005), no.~2, 487--530.



\bibitem{Lusztig} G.~Lusztig, \textit{Canonical bases arising from quantized enveloping
algebras}, J. Amer. Math. Soc. \textbf{3} (1990), 447-498.
\bibitem{Lusztig2000} G.~Lusztig, \textit{Constructible functions on varieties attached to quivers}, in {Studies in memory of
Issai Schur}, 177--223, Progress in Mathematics  {\bf 210},
Birkh\"auser 2003.



\bibitem{Nakajima1998} H.~Nakajima, \textit{Quiver varieties and Kac-Moody algebras}, Duke Math. J. {\bf 91} (1998), 515--560.




\bibitem{PX1997} L.~Peng, J.~Xiao, \emph{Root categories and simple Lie algebras}, J.
Algebra. {\bf{198}} (1997), 19--56.


\bibitem{PX2000} L.~Peng and J.~Xiao, \emph{Triangulated categories and Kac-Moody algebras}, Invent. Math. {\bf 140} (2000), 563--603.



\bibitem{Ringel1990} C.~M.~Ringel, \emph{Hall algebras and quantum groups}, Invent. Math. {\bf 101} (1990), 583--592.



\bibitem{Schiffmann1} O.~Schiffmann, \emph{Lectures on Hall
algebras}, arXiv: math/0611617.

\bibitem{Schiffmann2} O.~Schiffmann, \textit{Lectures on canonical and crystal bases of Hall
algebras}, arXiv: 0910.4460.

\bibitem{SX} J.~Sheng, F.~Xu, \emph{Derived Hall algebras and lattice algebras}, to appear in Algebra
Colloquium.


\bibitem{Steinitz} E.~Steinitz, \emph{Zur Theorie der Abel'schen Gruppen} Jahresberichts der DMV 9 (1901), 80--85.

\bibitem{Toen2005} B.~To\"en, \textit{Derived Hall algebras}, Duke Math. J. {\bf 135} (2006), no. 3, 587--615.

\bibitem{Xiao95} J.~Xiao, \emph{Hall algebra in a root category}, SFB preprint 95-070,
Bielefeld University, 1995.
\bibitem{Xiao97} J. Xiao, \textit{Drinfeld double and Ringel-Green theory of Hall algebras}, J. Algebra {\bf 190} (1997), no. 1,
100--144.
\bibitem{XX2006} J. Xiao and F. Xu, \textit{Hall algebras associated to triangulated categories}, Duke Math. J. {\bf 143} (2008),
no. 2, 357--373.
\bibitem{Xu2007} F. Xu, \textit{Hall algebras associated to triangulated categories, II: almost
associativity}, arXiv:0710.5588.
\end{thebibliography}

\end{document}